\documentclass[11pt]{article}
\usepackage{graphicx, color}
\usepackage[margin=1truein]{geometry}
\usepackage{amsmath,amsfonts,amssymb,amsthm,mathtools}
\usepackage{autobreak}
\usepackage{cite}
\usepackage{float}
\usepackage{indentfirst}
\usepackage{url}
\usepackage[colorlinks,linkcolor=red,citecolor={blue}]{hyperref}
\def\X{\mathbb{X}}

\newtheorem{thm}{Theorem}

\newtheorem{exa}{Example}

\newtheorem{cor}{Corollary}
\newtheorem{algo}{Algorithm}

\numberwithin{equation}{section}
\allowdisplaybreaks[4]

\begin{document}
    \title{Strong Polynomiality of the Value Iteration Algorithm for Computing Nearly Optimal Policies for Discounted Dynamic Programming}
    \author{Eugene A. Feinberg\footnote{Corresponding author}\\ \emph{Department of Mathematics and Statistics, Stony Brook University}\\ \emph{Stony Brook, NY 11794-3600, USA}\\Eugene.Feinberg@stonybrook.edu\\\\ Gaojin He\\ \emph{Department of Mathematics, University of California, San Diego}\\\emph{La Jolla, CA 92093-0112, USA}\\Ghe@ucsd.edu}
    \date{\today}
    \maketitle
\begin{abstract}
\noindent This note provides  upper bounds on the number of operations required to compute by value iterations a nearly optimal policy for an infinite-horizon discounted Markov decision process with a finite number of states and actions.  For a given discount factor,  magnitude of the reward function, and desired closeness to optimality, these upper bounds are strongly polynomial in the number of state-action pairs, and one of the provided upper bounds has the property that it is a non-decreasing function of the value of the discount factor.

\end{abstract}
\section{Introduction}
Value and policy iteration algorithms are the major tools for solving infinite-horizon discounted Markov decision processes (MDPs). Policy iteration algorithms also can be viewed as implementations of specific versions of the simplex method applied to linear programming problems corresponding to  discounted MDPs \cite{LCM,Ye}. Ye \cite{Ye} proved that for a given discount factor the policy iteration algorithm is strongly polynomial as a function of the total number of state-action pairs. Kitahara and Mizuno\cite{Kita} extended Ye's~\cite{Ye} results by providing sufficient conditions for strong polynomiality of a simplex method for linear programming, and Scherrer~\cite{Scherrer} improved Ye's~\cite{Ye} bound for MDPs.  For deterministic MDPs Post and Ye \cite{YePost} proved that for policy iterations there is a polynomial upper bound on the number of operations, which does not depend on the value of the discount factor. Feinberg and Huang~\cite{WVI} showed that value iterations are not strongly polynomial for discounted MDPs.  Earlier Tseng~\cite{Tseng} proved weak polynomiality of value iterations.

In this note we show that the value iteration algorithm computes $\epsilon$-optimal policies in strongly polynomial time. This is an important observation because value iterations are broadly used in applications,  including reinforcement learning~\cite{BerR, Ber, Sut}, for computing nearly optimal policies.

Let us consider an MDP with a finite state space $\X=\{1,2,\ldots, m\}$ and with a finite nonempty sets of actions $A(x)$ available at each state $x\in\X.$  Each action set $A(x)$ consists of $k_x$ actions.   Thus, the total number of actions is $k=\sum_{x=1}^m k_x,$ and this number can be interpreted as the total number of state-action pairs.  Let $\alpha\in [0,1)$ be the discount factor.  According to \cite{Scherrer}, the policy iteration algorithm finds an optimal policy within $N^{PI}_I(\alpha)$ iterations, and each iterations requires at most $N^{PI}_O$ operations with
\begin{equation}\label{epiab}
N^{PI}_I(\alpha):=O\left(\frac{k}{1-\alpha}\log\frac{1}{1-\alpha}\right),\qquad\qquad N^{PI}_O:=O\left(m^3+mk\right).
\end{equation}
This paper shows that for each $\epsilon>0$, the value iteration algorithm finds an $\epsilon$-optimal policy within $N^{VI}_I(\epsilon)$ iterations, and each iterations requires at most $N^{VI}_O$ operations, where 
\begin{equation}\label{eubce12}
N^{VI(\epsilon)}_I(\alpha):=\max\left\{\left\lceil\frac{\log\frac{\left(1-\alpha\right)\epsilon}{\left[R+(1+\alpha)V\right]}}{\log\alpha}\right\rceil,1\right\}, \qquad\qquad N^{VI}_O:=O\left(mk\right),
\end{equation}
where $R$ and $V$ are the constants defined by a one-step cost function $r$ and a function of terminal rewards $v^0$. In addition, $N^{VI(\epsilon)}_I(\alpha)$ is non-decreasing in $\alpha\in [0,1)$.

To define the values of $R$ and $V,$   let us denote by $sp(u)$  the span seminorm of $u\in\mathbb{R}^m,$ where $\mathbb{R}^m$ is an $m$-dimensional Euclidean space,
\begin{equation*}
sp(u):=\max_{x\in\mathbb{X}}u(x)-\min_{x\in\mathbb{X}}u(x).
\end{equation*}
The properties of this seminorm can be found in \cite[pp.~196]{Puterman}.

Let $r(x,a)$ be the reward collected if the system is at the state $x\in \X$ and the action $a\in A(x)$ is chosen.  Let $v^0(x)$ be the reward collected at the final state $x\in\X.$  Then \begin{equation}\label{ev0ef} R:=sp(v^0).\end{equation}  We denote by
\begin{equation}\label{defv1}
v^1(x):=\max_{a\in A(x)} r(x,a),\qquad x\in\X,
\end{equation}
the maximal one-step reward that can be collected at the state $x.$ Then
\begin{equation}\label{ev1ef}
V:=sp(v^1)= \max_{x\in\X}\max_{a\in A(x)} r(x,a) - \min_{x\in\X}\max_{a\in A(x)} r(x,a).
\end{equation}
Observe that
\[ V\le sp(r):= \max_{x\in\X}\max_{a\in A(x)} r(x,a) - \min_{x\in\X}\min_{a\in A(x)} r(x,a).\]

 Of course, the total number of operations to find an $\epsilon$-optimal policy   is bounded above by $N^{VI(\epsilon)}_I(\alpha)\cdot N^{VI}_O$ for the value iteration algorithm. The total number of operations to find an optimal policy   is bounded above by $N^{PI}_I(\alpha)\cdot N^{PI}_O$ for the policy iteration algorithm.    Each iteration in the policy iteration algorithm requires solving a system of $m$ linear equations.  This can be done by Gaussian elimination within $O(m^3)$ operations.  This is the reason the formula for $N^{PI}_O$ in \eqref{epiab} depends on the term $m^3,$ which can be reduced to $m^\omega$ with $\omega < 3$ by using contemporary methods.  For example, $\omega=2.807$ for Strassen's algorithm\cite{le}. According to \cite{Tensor}, the best currently available $\omega=2.37286,$  but this method of solving linear equations is impractical due to the large value of the constant in $O.$ In addition to \eqref{eubce12}, more accurate upper bounds on a number of iterations are presented in \eqref{bound1}, \eqref{bound2}, and \eqref{bound3}, but they require additional definitions and calculations.

As well-known and clear from \eqref{epiab} and \eqref{eubce12}, the number of operations at each step is larger for the policy iteration algorithm than for the value iteration algorithm.  If the number of states $m$ is large, then the difference $(N^{PI}_O-N^{VI}_O)$ can be significant.  In order to accelerate policy iterations, the method of modified policy iterations was introduced in \cite{MPI}. This method uses value iterations to solve linear equations. As shown in Feinberg et al.~\cite{WMPI}, modified policy iterations and their versions are not strongly polynomial algorithms for finding optimal policies.

\section{Definitions}
Let $\mathbb{N}$ and $\mathbb{R}$ be the sets of natural numbers and real numbers respectively. For a finite set $E$, let $\left|E\right|$ denote the number of elements in the set $E$. We consider an MDP with a finite state space $\mathbb{X}=\{1,2,\ldots, m\},$ where $m\in\mathbb{N}$ is the number of states, and nonempty finite action sets $A(x)$ available at  states $x\in\X.$ Let $\mathbb{A}:=\bigcup_{x\in\mathbb{X}}A(x)$ be the action set. We recall that  $k=\sum_{x\in\X}\left|A(x)\right|$ is the total number of actions at all states or, in slightly different terms, the number of all state-action pairs. For each $x\in\mathbb{X},$ if an action $a\in A(x)$ is selected at the state $x\in\X,$ then a one-step reward $r(x,a)$ is collected and the process moves to the next state $y\in\X$ with the probability $p(y|x,a),$  where $r(x,a)$ is a real number and $\sum_{y\in\X} p(y|x,a)=1.$   The process continues over a finite or infinite planning horizon. For a finite-horizon problem, the terminal real-valued reward $v^0(x)$ is collected at the final state $x\in\X.$

A \textit{deterministic policy} is a mapping $\phi:\mathbb{X}\mapsto\mathbb{A}$ such that $\phi(x)\in A(x)$ for each $x\in\mathbb{X},$ and, if the process is at a state $x\in\X,$ then the action $\phi(x)$ is selected.  A general policy can be randomized and history-dependent; see e.g., Puterman~\cite[p.~154]{Puterman} for definitions of various classes of policies. We denote by $\Pi$  the set of all policies. 


Let $\alpha\in[0,1)$ be the \textit{discount factor}. For a policy $\pi\in\Pi$ and for an  initial state $x\in\X,$ the expected total discounted reward for an $n$-horizon  problem is
\begin{equation*}
v_{n,\alpha}^{\pi}(x) :=\mathbb{E}_x^\pi\left[\sum_{t=0}^{n-1}\alpha^kr(x_t,a_t)+\alpha^nv^0(x_n)\right],\qquad n\in\mathbb{N},
\end{equation*}
and for the infinite-horizon problem it is
\begin{equation*}
v_\alpha^{\pi}(x) :=\mathbb{E}_x^\pi\sum_{t=0}^{\infty}\alpha^kr(x_t,a_t),
\end{equation*}
where $\mathbb{E}_x^\pi$ is the expectation defined by the initial state $x$ and the policy $\pi,$ and  where $x_0a_0\cdots x_{t}a_{t}  \ldots$ is a  trajectory of the process consisting of states $x_t\in\X$ and actions $a_t\in A(x_t)$ at epochs $t=0,1,\ldots\ .$
The \textit{value functions} are defined for initial states $x\in\X$ as
\begin{equation*}
v_{n,\alpha}(x) : = \sup_{\pi\in\Pi}v_{n,\alpha}^\pi(x), \qquad n\in\mathbb{N}, 
\end{equation*}
for $n$-horizon problems,   and
\begin{equation*}
v_\alpha(x) := \sup_{\phi\in\mathbb{F}}v_\alpha^\phi(x)
\end{equation*}
for infinite-horizon problems.  Note that $v_{0,\alpha}^\pi=v_{0,\alpha}=v^0$ for all $\alpha\in[0,1)$ and $\pi\in\Pi$.

A policy $\pi$ is called \textit{optimal} ($n$-\textit{horizon optimal} for $n=1,2,\ldots$) if $v_\alpha^\pi(x)=\sup_{\pi\in\Pi}v_\alpha^\pi(x)$   ($v_{n,\alpha}^\pi(x)=\sup_{\pi\in\Pi}v_{n,\alpha}^\pi(x)$) for all $x\in\mathbb{X}$. It is well-known that, for  discounted MDPs with finite action sets,  finite horizon problems have optimal policies and  infinite horizon problems have deterministic optimal policies; see \cite[p.~154]{Puterman}.

 A policy $\pi$ is called $\epsilon$-optimal for $\epsilon\ge 0$ if $v_{\alpha}^\pi (x)\ge v_{\alpha}(x)-\epsilon$ for all $x\in\X.$  A $0$-optimal policy is optimal.   The objective of this paper is to estimate the complexity of the value iteration algorithm for finding a deterministic $\epsilon$-optimal policy for  $\epsilon>0$.

\section{Main Results}

For a real-valued function $v:\X\to\mathbb{R},$ let us define
\begin{equation}\label{oper}
T_\alpha^a v(x):= r(x,a)+ \alpha\sum_{y\in\mathbb{X}}p(y|x,a)v(y), \qquad x\in\X,\ a\in A(x).
\end{equation}
We shall use the notation $T_\alpha^\phi v(x):= T_\alpha^{\phi(x)} v(x)$ for a deterministic policy $\phi.$  For $v:\X\to\mathbb{R}$ we also define the optimality operator $T_\alpha,
$
\begin{equation}\label{operm}
T_\alpha v(x):= \max_{a\in A(x)} T_\alpha^a v(x), \qquad x\in\X.
\end{equation}
Every real-valued function $v$ on $\X$ can be identified with a vector $v=(v(1),\ldots,v(m)).$  Therefore, all real-valued functions on $\X$ form the $m$-dimensional
Euclidean space $\mathbb{R}^m.$

For each $n\in\mathbb{N}$ and all $v^0\in\mathbb{R}^m$, the expected total discounted rewards $v_{n,
\alpha}^\phi$ and $v_{\alpha}^\phi$ satisfy the equations
\begin{equation*}
v^\phi_{n,\alpha}  = T^\phi_\alpha v^\phi_{n-1,\alpha}, \quad v^\phi_{\alpha}  = T^\phi_\alpha v^\phi_{\alpha},
\end{equation*}
the value functions $v_{n,\alpha},$ $v_\alpha$  satisfy the \textit{optimality equations}
\begin{equation*}
v_{n,\alpha}  = T_\alpha v_{n-1,\alpha}, \quad v_{\alpha}  = T_\alpha v_{n-1,\alpha},
\end{equation*}
and \emph{value iterations} converge to the infinite-horizon expected total rewards and optimal values
\begin{equation}\label{op}
\begin{aligned}
v_{\alpha}^\phi = & \lim_{n\to\infty}{\left(T_\alpha^\phi\right)}^nv^0 = T_\alpha^\phi v_\alpha^\phi,\\
v_{\alpha} = \lim_{n\to\infty}{\left(T_\alpha\right)}^n&v^0 = \lim_{n\to\infty}v_{n,\alpha} = T_\alpha v_\alpha,
\end{aligned}
\end{equation}
and a deterministic policy $\phi$ is optimal if and only if $T^\phi_\alpha v_\alpha =  T_\alpha v_\alpha;$ see, e.g., \cite{Fhand}, \cite[pp.~146-151]{Puterman}. Therefore,
if we consider the nonempty sets
\begin{equation*}
 A_{\alpha}(x)= \{a\in A(x): v_{\alpha}(x)=T^a_\alpha v_{\alpha}(x)\},
 \end{equation*}
 then a deterministic policy $\phi$ is optimal if and only if $\phi(x)\in A_{\alpha}(x)$ for all $x\in\X.$

For a given $\epsilon>0,$ the following value iteration algorithm computes a deterministic $\epsilon$-optimal policy.  It uses a stopping rule based on the value of the span $sp(v_{n,\alpha}-v_{n-1,\alpha})$. As mentioned in \cite[p. 205]{Puterman}, this algorithm generates the same number of iterations as the relative value iteration algorithm originally introduced to accelerate value iterations.

%
%
\begin{algo}\label{algo}(Computing  a deterministic $\epsilon$-optimal policy by value iterations).\rm{
\begin{align*}
& \leftline{For given MDP, discount factor $\alpha\in (0,1),$ and constant $\epsilon>0$:}\\
& \leftline{1.  select a vector $u:=v^0\in\mathbb{R}^m$ and a constant $\Delta> \frac{1-\alpha}{\alpha}\epsilon$ (e.g., choose $\Delta:=\epsilon/\alpha$);}\\
& \leftline{2. while $\Delta>\frac{1-\alpha}{\alpha}\epsilon$ compute $v=T_\alpha u,$ $\Delta:=sp(u-v)$ and set $u^*:= u,$ $u:=v$ endwhile;}\\  
& \leftline{3. choose a deterministic policy $\phi$ such that $v=T_\alpha^\phi u^*;$ this policy is $\epsilon$-optimal.}
\end{align*} }
\end{algo}

If $\alpha=0,$ then a deterministic policy $\phi$ is optimal if and only if $r(x,\phi(x))=\max_{a\in A(x)}\{r(x,a)\}.$
As well-known, Algorithm~\ref{algo} converges within a finite number of iterations (e.g., this follows from  \eqref{op}) and returns an $\epsilon$-optimal policy $\phi$ (e.g., this follows from \cite[Proposition 6.6.5]{Puterman}).

Following Puterman~\cite[Theorem 6.6.6]{Puterman}, let us define
\begin{equation}\label{gamma}
\gamma:= \max_{\substack{x,y\in\mathbb{X}\\ a\in A(x),b\in A(y)}}\left[1-\sum_{z\in\mathbb{X}}\min\left\{p(z|x,a),p(z|y,b)\right\}\right].
\end{equation}
We notice that $0\le\gamma\le1$. We assume $\gamma>0$ since otherwise $p(z|x,a)=p(z|y,b)$ for any $x,y,z\in\mathbb{X}$ and $a\in A(x),b\in A(y).$ If $\gamma =0,$ then we have again that a deterministic policy $\phi$ is optimal if and only if $r(x,\phi(x))=\max_{a\in A(x)}\{r(x,a)\}.$ In this case Algorithm~\ref{algo} stops at most after the second iteration  and returns a deterministic optimal policy. If an MDP has deterministic probabilities and there are two or more deterministic policies with nonidentical stochastic matrices, then $\gamma=1$.

Computing the value of $\gamma$ requires computing the sum in \eqref{gamma} for all couples $\{(x,a),(y,b))\}$  of state-action pairs such that $(x,a)\ne(y,b).$ The identical pairs $(x,a)=(y,b)$ can be excluded since for such pairs the value of the function to be maximized in \eqref{gamma} is 0, which is the smallest possible value.   The total number of such couples is $k(k-1)/2=O(k^2)$.   The number of arithmetic operations in \eqref{gamma}, which are additions, is $m$ for each couple.  Therefore, the straightforward computation of  $\gamma$ requires $O(mk^2)$ operations, which can be significantly larger than the complexity to compute an deterministic $\epsilon$-optimal policy, which is the product of $N^{VI(\epsilon)}_I(\alpha)$ and $N^{VI}_O$ defined in \eqref{eubce12}. Puterman~\cite[Eq. (6.6.16)]{Puterman} also  provides an upper bound $\gamma'\in [\gamma,1],$ where $\gamma':= 1-\sum_{z\in\X} \min_{x\in\X, a\in A(x)} p(z|x,a),$  whose computation requires $O(mk)$ operations.

\begin{thm}\label{th1}
For $\alpha\in(0,1)$ and $\epsilon>0$, the number of iterations for Algorithm~\ref{algo} to find a deterministic $\epsilon$-optimal  policy is bounded above by
\begin{equation}\label{bound1}
n^*(\alpha):=\max\left\{\left\lceil\frac{\log\frac{(1-\alpha)\epsilon\gamma}{sp(v_{1,\alpha}-v^0)}}{\log(\alpha\gamma)}\right\rceil,1\right\}
\end{equation}
if $sp(v_{1,\alpha}-v^0)\ne0$ and $\gamma>0.$ If $sp(v_{1,\alpha}-v^0)=0,$ then $n^*(\alpha)=1$ and Algorithm~\ref{algo} returns a deterministic optimal policy. If $\gamma=0,$ then $n^*(\alpha)\in \{1,2\}$  and Algorithm~\ref{algo}  returns a deterministic optimal policy. In addition, each iteration uses at most $O(mk)$ operations.
\end{thm}
\begin{proof}As follows from \eqref{oper} and \eqref{operm}, each iteration uses at most $O(mk)$ arithmetic operations. Let $n$ be the actual number of iterations. According to its steps 2 snd 3, the algorithm returns the deterministic policy $\phi,$ for which $v_{n,\alpha}=T_\alpha v_{n-1,\alpha}=T_\alpha^\phi v_{n-1,\alpha}$ and
\begin{equation*}
sp(T_\alpha v_{n-1,\alpha}-v_{n-1,\alpha})\le\frac{1-\alpha}{\alpha}\epsilon.
\end{equation*}
In view of \cite[Proposition~6.6.5, p.~201]{Puterman}, this $\phi$ is $\epsilon$-optimal. By \cite[Corollary~6.6.8, p.~204]{Puterman},
\begin{equation}\label{b}
sp(v_{n,\alpha}-v_{n-1,\alpha})\le\left(\alpha\gamma\right)^{n-1}sp(v_{1,\alpha}-v^0).
\end{equation}
Therefore, the minimal number $n\in\mathbb{N}$ satisfying
\begin{equation}\label{c}
\left(\alpha\gamma\right)^{n-1}sp(v_{1,\alpha}-v^0)\le\frac{1-\alpha}{\alpha}\epsilon
\end{equation}
leads to \eqref{bound1}. The case $sp(v_{1,\alpha}-v^0)=0$ is obvious, and the case $\gamma=0$ is discussed above.
\end{proof}

\begin{exa}\label{ex1} This example shows that the bound in \eqref{bound1} can be exact, and it may not be monotone in the discount factor. {\rm   Let the state space be $\mathbb{X}=\left\{1,2,3\right\}$ and the action space be $\mathbb{A}=\left\{b,c\right\}$. Let $A(1)=\mathbb{A}$, $A(2)=A(3)=\left\{b\right\}$ be the sets of actions available at states $1,$ $2,$ and $3$ respectively. The transition probabilities are given by $p(3|1,b)=p(2|1,c)=p(2|2,b)=p(3|3,b)=1$. The one-step rewards are  $r(1,b)=r(1,c)=0, r(2,b)=1$, and $r(3,b)=-1;$ see \autoref{fig:2}.

\begin{figure}[h!]
\centerline{\includegraphics[scale=1.2]{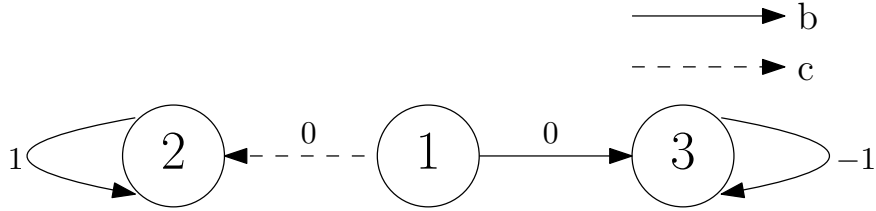}}
\caption{MDP diagram for Example \ref{ex1}.}
\label{fig:2}
\end{figure}
We set $v^0(1)=1,v^0(2)=2,v^0(1)=-2$. As discussed above,  $\gamma=1$ for this MDP with deterministic transitions. Straightforward calculations imply that
\begin{equation*}
\begin{aligned}
&v_{n,\alpha} =
\left(
\alpha^n+\sum_{k=1}^n\alpha^k,\ \alpha^n+\sum_{k=0}^n\alpha^k,\ -\alpha^n-\sum_{k=0}^n\alpha^k
\right),\\
&v_{n,\alpha}-v_{n-1,\alpha}=\left(2a^n-a^{n-1},\ 2a^n-a^{n-1},\ -2a^n+a^{n-1} \right),\\
&sp(v_{n,\alpha}-v_{n-1,\alpha}) = 2\alpha^{n-1}\left|2\alpha-1\right|=\alpha^{n-1}sp(v_{1,\alpha}-v^0),
\end{aligned}
\end{equation*}
where the $i$-th coordinates of the vectors correspond to the states $i=1,2,3$. The last displayed equality implies that inequality \eqref{b} holds in the form of an equality for this example.  Therefore, the bound in \eqref{bound1} is also the actual number of iterations executed by Algorithm~\ref{algo} for  this MDP, which is
\begin{equation*}
n^*(\alpha)=\max\left\{\left\lceil\frac{\log\frac{(1-\alpha)\epsilon}{2\left|2\alpha-1\right|}}{\log\alpha}\right\rceil,1\right\}.
\end{equation*}
for $\alpha\ne0.5$ and $\epsilon>0$. If $\alpha=0.5,$ then Algorithm~\ref{algo} stops after the first iteration. Let $\epsilon=0.02.$ Then  $n^*(0.24)=3,$ $n^*(0.47)=4,$ and $n^*(0.48)=3$, which shows that $n^*(\alpha)$ may not be monotone in $\alpha$.  \hfill $\square$ }
\end{exa}
Let us consider the vector $v^1\in\mathbb{R}^m$ defined in \eqref{defv1}.
Note that $v^1=v_{1,\alpha}$ if $v^0\equiv 0$. In the next corollary, we assume that $sp(v^1)+sp(v^0)>0$ and $\gamma>0$ since otherwise Algorithm~\ref{algo} stops after the first iteration and returns a deterministic optimal policy.
\begin{thm}
Let $\alpha\in(0,1)$. For fixed $\epsilon>0,\gamma\in(0,1]$, and $v^0,v^1\in\mathbb{R}^m$ such that $sp(v^1)+sp(v^0)>0$,  Algorithm~\ref{algo}  finds a deterministic $\epsilon$-optimal policy after the finite number iteration  bounded above by
\begin{equation}\label{bound2}
F(\alpha):=\max\left\{\left\lceil\frac{\log\frac{(1-\alpha)\epsilon\gamma}{sp(v^1)+(1+\alpha)sp(v^0)}}{\log(\alpha\gamma)}\right\rceil,1\right\}.
\end{equation}
Furthermore,  the function $F(\alpha)$ defined in \eqref{bound2} for $\alpha\in [0,1)$  has the following properties for an arbitrary fixed parameter $\gamma\in (0,1]$:
\begin{flalign*}
&\emph{(a)}\ \lim_{\alpha\downarrow0}F(\alpha)=1,\ \lim_{\alpha\uparrow1}F(\alpha)=+\infty;\nonumber&&\\
&\emph{(b)}\ \text{$F(\alpha)$ is non-decreasing in $\alpha$.}\nonumber&&
\end{flalign*}
\end{thm}
\begin{proof}
By \eqref{oper} and \eqref{operm},
\begin{equation*}
\begin{aligned}
v_{1,\alpha}(x) & =\max_{a\in A(x)}\left\{r(x,a)+ \alpha\sum_{y\in\mathbb{X}}p(y|x,a)v^0(y)\right\}\\
& \le  \max_{a\in A(x)}\left\{r(x,a)+ \alpha\sum_{y\in\mathbb{X}}p(y|x,a)\max_{z\in\mathbb{X}}v^0(z)\right\}\\
& = \max_{a\in A(x)}r(x,a)+\alpha\max_{z\in\mathbb{X}}v^0(z).
\end{aligned}
\end{equation*}
Similarly,
$
v_{1,\alpha}(x) \ge \max_{a\in A(x)}r(x,a)+\alpha\min_{z\in\mathbb{X}}v^0(z).
$
Therefore,
\begin{equation*}
\begin{aligned}
sp(v_{1,\alpha}) & = \max_{x\in\mathbb{X}}v_{1,\alpha}(x)-\min_{x\in\mathbb{X}}v_{1,\alpha}(x)\\
& \le \max_{x\in\mathbb{X}}\max_{a\in A(x)}r(x,a)+\alpha\max_{z\in\mathbb{X}}v^0(z)-\min_{x\in\mathbb{X}}\max_{a\in A(x)}r(x,a)-\alpha\min_{z\in\mathbb{X}}v^0(z)\\
& =sp(v^1)+\alpha sp(v^0).
\end{aligned}
\end{equation*}
By the properties of seminorm provided in \cite[p.~196]{Puterman},
\begin{equation*}
sp(v_{1,\alpha}-v^0)\le sp(v_{1,\alpha})+sp(v^0)\le sp(v^1)+(1+\alpha)sp(v^0),
\end{equation*}
which leads to $F(\alpha)\ge n^*(\alpha),$ where $n^*(\alpha)$ is defined in \eqref{bound1}.\\
\indent The formulae in (a) follow directly from \eqref{bound2}. To prove (b), we recall that $R=sp(v^1)$ and $V=sp(v^0);$ see \eqref{ev1ef} and \eqref{ev0ef}. By the assumption in the theorem, $R+(1+\alpha)V>0.$
 The function $\frac{(1-\alpha)\epsilon\gamma}{R+(1+\alpha)V}$ is positive and decreasing in $\alpha\in (0,1),$ and the function $\alpha\gamma$ is increasing in $\alpha\in (0,1).$  In addition, $\alpha\gamma\in (0,1).$  Therefore, the function
 \[f(\alpha):=\frac{\log\frac{(1-\alpha)\epsilon\gamma}{R+(1+\alpha)V}}{\log(\alpha\gamma)}\] is increasing when $\frac{(1-\alpha)\epsilon\gamma}{R+(1+\alpha)V}\le 1$, which is $\alpha\ge\frac{\epsilon\gamma-R-V}{\epsilon\gamma+V}$. This implies that $F(\alpha)$ is increasing on the interval  $(b,1),$ where $b:=\max\{\frac{\epsilon\gamma-R-V}{\epsilon\gamma+V},0\}$.  Thus, if $\frac{\epsilon\gamma-R-V}{\epsilon\gamma+V}\le 0,$ then the theorem is proved. Now let $\frac{\epsilon\gamma-R-V}{\epsilon\gamma+V}> 0.$  We choose an arbitrary $\alpha\in (0, \frac{\epsilon\gamma-R-V}{\epsilon\gamma+V}].$  Then $\frac{(1-\alpha)\epsilon\gamma}{R+(1+\alpha)V}\ge1,$ which implies $f(\alpha)\le0.$  In view of \eqref{bound2}, $F(\alpha)=1,$ and the function $F(\alpha)$ is non-decreasing on $(0,1).$\end{proof}
As explained in the paragraph preceding Theorem~\ref{th1}, it may be  time-consuming to find the actual value of $\gamma$ for an MDP.  The following corollaries provide additional bounds.
\begin{cor}
Let $\alpha\in(0,1)$. For a fixed $\epsilon>0,$  if $sp(v^1)+sp(v^0)>0$, then 
\begin{equation}\label{bound3}
n^*(\alpha)\le\max\left\{\left\lceil\frac{\log\frac{(1-\alpha)\epsilon}{sp({v_{1,\alpha}-v^0})}}{\log\alpha}\right\rceil,1\right\},
\end{equation}
\begin{equation}\label{bound4}
F(\alpha)\le N^{VI(\epsilon)}_I(\alpha):=\max\left\{\left\lceil\frac{\log\frac{(1-\alpha)\epsilon}{sp(v^1)+(1+\alpha)sp(v^0)}}{\log\alpha}\right\rceil,1\right\}.
\end{equation}
\end{cor}
\begin{proof}
We recall that $n^*(\alpha)$ is defined as the smallest $n\in\mathbb{N}$ satisfying  \eqref{c}.  The right-hand side of \eqref{bound3} the smallest $n\in\mathbb{N}$ satisfying $\left(\alpha\gamma\right)^{n-1}sp(v_{1,\alpha}-v^0)\le\frac{1-\alpha}{\alpha}\epsilon.$ Since $\alpha,\gamma\in(0,1),$ inequality \eqref{bound3} holds.  Inequality \eqref{bound4} holds because of the similar reasons, where $F(\alpha)$ and $N^{VI(\epsilon)}_I(\alpha)$ are the smallest $n\in\mathbb{N}$ satisfying $\left(\alpha\gamma\right)^{n-1}[sp(v^1)+sp(v^0)]\le\frac{1-\alpha}{\alpha}\epsilon$ and $\alpha^{n-1}[sp(v^1)+sp(v^0)]\le\frac{1-\alpha}{\alpha}\epsilon$ respectively.
\end{proof}

We  notice that, if the function $F$ from \eqref{bound2} is minimized in $v^0$, than the smallest value is attained when $sp(v^0)=0,$ which is $v^0=const.$ The following corollary provides upper bounds for $v^0=const$ including $v^0\equiv 0.$
\begin{cor}\label{cor2} Let $\alpha\in(0,1),$ and let $v^0=const.$  If $sp(v^1)>0,$ then  Algorithm~\ref{algo}  finds a deterministic $\epsilon$-optimal policy after a finite number of iterations bounded above by
\begin{equation*}
F^*(\alpha):=\max\left\{\left\lceil\frac{\log\frac{(1-\alpha)\epsilon\gamma}{sp(v^1)}}{\log(\alpha\gamma)}\right\rceil,1\right\}\le\max\left\{\left\lceil\frac{\log\frac{(1-\alpha)\epsilon}{sp(v^1)}}{\log\alpha}\right\rceil,1\right\}.
\end{equation*}
\end{cor}
\begin{proof}
This corollary follows from \eqref{bound4}.
\end{proof}

\begin{exa}\label{ex2} This example illustrates the monotonicity of the polynomial upper bound $F(\alpha)$ for computing $\epsilon$-optimal policies and non-monotonicity of the number of calculations to find an optimal policy by value iterations.  {\rm The    following MDP is taken from \cite{WVI}.  Let the state space be $\mathbb{X}=\left\{1,2,3\right\}$ and the action space be $\mathbb{A}=\left\{b,c\right\}$. Let $A(1)=\mathbb{A}$, $A(2)=A(3)=\left\{b\right\}$ be the sets of actions available at states $1,2,3$ respectively; see  \autoref{fig:1EF}. The transition probabilities are given by $p(2|1,c)=p(3|1,b)=p(2|2,b)=p(3|3,b)=1$. The one-step rewards are $r(1,b)=r(2,b)=0, r(3,b)=1$, and $r(1,c)=1-\exp{(-M)}$ where $M>0$. We set $v^0(x)=0$ for $x\in\mathbb{X}$.
\begin{figure}[H]
\centerline{\includegraphics[scale=1.2]{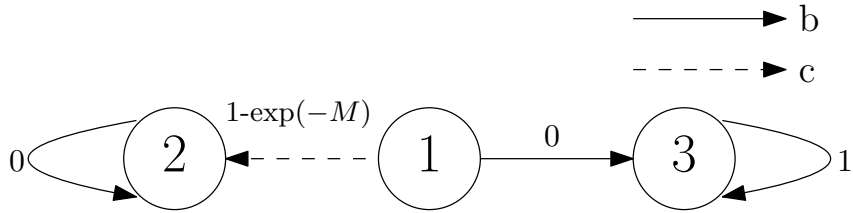}}
\caption{MDP diagram for Example \ref{ex2}.}
\label{fig:1EF}
\end{figure}
As shown in \cite{WVI}, for $\alpha=0.5$ the number of value iterations  required  to find an optimal policy  increases to infinity as $M$ increases to infinity. This shows that the value iteration algorithm for computing the optimal policy is not strongly polynomial. However, $sp(v^1)=1$ does not change with the increasing $M$ in this example. As follows from Corollary~\ref{cor2}, for fixed $\epsilon>0$ and $\alpha\in (0, \alpha^*]$ with $\alpha^*\in (0,1),$ the number of required iterations $N^{VI(\epsilon)}_I$ for Algorithm~\ref{algo} is uniformly bounded no matter how large  $M$ is; see  \autoref{fig:1}.\ \hfill$\square$}
\begin{figure}
\centerline{\includegraphics[scale=0.43]{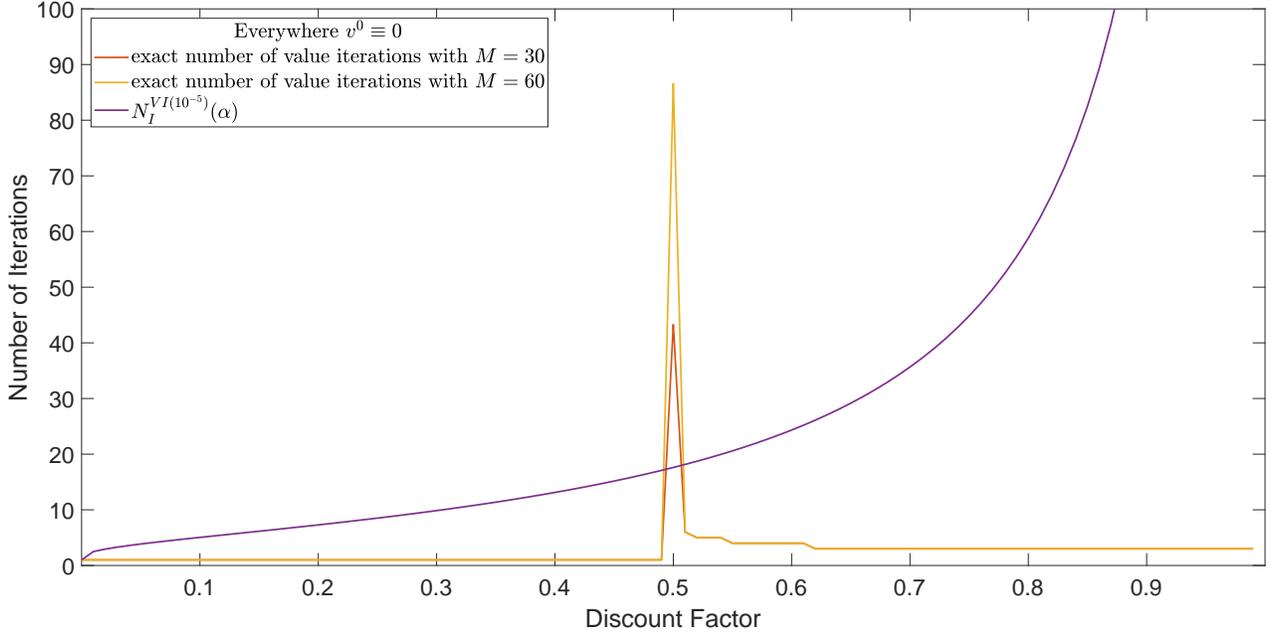}}
\caption[]{Exact numbers of iterations for finding optimal policies and $N^{VI(10^{-5})}_I$ in Example~\ref{ex2}.\footnotemark}
\label{fig:1}
\end{figure}
\end{exa}
Lewis and Paul \cite{TP} provide examples of MDPs for which the numbers of iterations required for computing an optimal policy by value iterations can tend to infinity even for discount factor bounded away from $1$. Here we provide a significantly simpler example.

\begin{exa}\label{ex3} This example shows that the required number of  value iterations to compute an optimal policy may be unbounded as a function of the discount factor $\alpha$ when $\alpha\in (0,\alpha^*]$ for some $\alpha^*\in (0,1).$  {\rm Consider an MDP with the state space  $\mathbb{X}=\left\{1,2,3\right\},$  with the action space $\mathbb{A}=\left\{b,c\right\},$ and with the sets of actions $A(1)=\mathbb{A}$ and $A(2)=A(3)=\left\{b\right\}$   available at states $1,$ $2,$ and $3$ respectively;  see \autoref{fig:213}. The transition probabilities are given by $p(2|1,c)=p(3|1,b)=p(2|2,b)=p(3|3,b)=1$. The one-step rewards are  $r(1,c)=r(2,b)=1, r(1,b)=2$, and $r(3,b)=0$.

We set $v^0(x)=0$ for $x\in\mathbb{X}$.
\begin{figure}[h!]
\centerline{\includegraphics[scale=1.2]{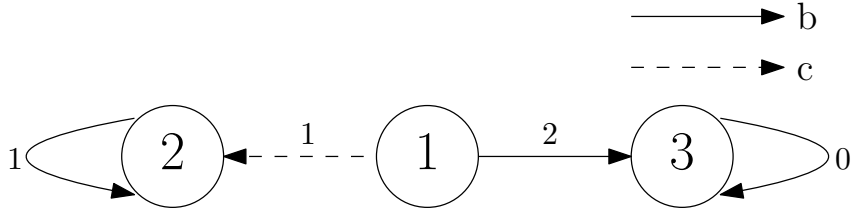}}
\caption{MDP diagram for Example \ref{ex3}.}
\label{fig:213}
\end{figure}
\footnotetext{Numbers of iterations are integers.  In particular, in this example the graphs display step functions. This graph is generated by the Matlab, and it is automatically smoothened to clarify comparisons.}
There are only two deterministic policies denoted by $\phi$ and $\psi,$ which  differ only at state $1$ with $\phi(1)=c$ and $\psi(1)=b$. Observe that $v_\alpha^\pi(x)=v_\alpha(x)$ for $x=2,3$ and all $\pi\in\Pi.$ In addition,
\[v_\alpha^\phi(1)=\sum_{k=0}^\infty\alpha^k=\frac{1}{1-\alpha}, \qquad v_\alpha^\psi(1)=2.\]
Therefore,  $\psi$ is optimal for $\alpha \in [0,0.5]$ and $\phi$ is optimal for $\alpha \in [0.5,1).$ Let us consider the number $\Delta$ computed at the $n$-th value iteration at step 2 of Algorithm~\ref{algo},  $n\in\mathbb{N}.$  For this MDP, $\Delta=v-u$ with $v=v_{n,\alpha}=T_\alpha u$ and $u=v_{n-1,\alpha}.$  Observe that $v_{n, \alpha}(3)=0$ and  $v_{n,\alpha}(2)=\sum_{i=0}^{n-1}\alpha^i$ for all $n\in\mathbb{N}.$ If $n=1,2,$ then $v_{n,\alpha}(1)=2$. For $n\ge3$
\[ v_{n,\alpha}(1)=\max\left\{1+\alpha\sum_{i=0}^{n-2}\alpha^i, 2+0\right\}=
\begin{cases}
2,\  &{\rm if\ }\alpha\in (0,0.5+\delta_n];\\
\sum_{i=0}^{n-1}\alpha^i, \ &{\rm if\ } \alpha\in [0.5+\delta_n,1),
\end{cases}
\]
where $0<\delta_n<0.5$ and $\sum_{i=0}^{n-1}{(0.5+\delta_n)}^i=2$. Clearly $\delta_n$ is strictly decreasing as $n$ increases and $\lim_{n\to\infty}\delta_n=0$. This implies that for $\alpha\in (0,0.5]$ the value iterations identify the optimal policy $\psi$ at the first iteration, and for every $n\ge3$ the value iterations identify the optimal policy $\phi$ at the $n$-th iteration for $\alpha\in(0.5+\delta_{n+1},0.5+\delta_n)$.\hfill$\square$}
\end{exa}

Example~\ref{ex2} and Example~\ref{ex3} represent the main difficulties of running value iterations for computing optimal policies. Nevertheless, the results of this paper show that these difficulties can be easily overcome by using  value iterations for computing $\epsilon$-optimal policies.

\section{Acknowledgement}
This research  was partially supported by the National Science Foundation grant CMMI-1636193.


\begin{thebibliography}{12}

\bibitem{BerR}  D. P. Bertsekas, Reinforcement Learning and Optimal Control,  Athena Scientific, Belmont, MA, 2019.

\bibitem{Ber}
D. P. Bertsekas, J. N. Tsitsiklis, Neuro-Dynamic Programming, Athena Scientific, Belmont, MA, 1996.

\bibitem{Fhand}
E. A. Feinberg, Total reward criteria,
in: E. A. Feinberg, A. Schwartz (Eds.), Handbook of Markov Decision Processes, Kluwer, Boston, MA, 2002, pp. 173-207.

\bibitem{WVI}
E. A. Feinberg and J. Huang, The value iteration algorithm is not strongly polynomial for discounted dynamic programming, Oper. Res. Lett. 42 (2014) 130-131.

\bibitem{WMPI}
E. A. Feinberg, J. Huang, and B. Scherrer, Modified policy iteration algorithms are not strongly polynomial for discounted dynamic programming, Oper. Res. Lett. 42 (2014) 429-431


\bibitem{LCM}
L.C.M. Kallenberg, Finite state and action MDPs, in: E. A. Feinberg, A. Schwartz (Eds.), Handbook of Markov Decision Processes, Kluwer, Boston, MA, 2002, pp. 21-87.

\bibitem{Kita}
T. Kitahara and S. Mizuno, A bound for the number of different basic solutions generated by the simplex method, Math. Program. 137 (2013) 579-586.

\bibitem{Tensor}
J. F. Le Gall, Powers of tensors and fast matrix multiplication, Proc. of 39-th Int. Symp. Symb. Alge. Comput. (2014) 296-303.

\bibitem{TP}
M. E. Lewis and A. Paul, Uniform turnpike theorems for finite Markov decision processes, Math. Oper. Res. 44 (2019) 1145-1160.

\bibitem{YePost}
I. Post and Y. Ye, The simplex method is strongly polynomial for deterministic Markov decision processes, Math. Oper. Res. 40 (2015) 859-868.

\bibitem{Puterman}
M. L. Puterman, Markov Decision Processes: Discrete Stochastic Dynamic Programming, John Wiley \& Sons, Inc., New York, 1994.

\bibitem{MPI}
M. L. Puterman and M. C. Shin, Modified policy iteration algorithms for discounted Markov decision problems, Manag. Sci. 24 (1978) 1127-1137.

\bibitem{Scherrer}
B. Scherrer, Improved and generalized upper bounds on the complexity of policy iteration, Math. Oper. Res. 41 (2016) 758-774.

\bibitem{le}
V. Strassen, Gaussian elimination is not optimal, Numer. Math. 13 (1969)  354-356.

\bibitem{Sut}
R. S. Sutton and  A. G. Barto,  Reinforcement Learning: An Introduction, second ed., MIT Press, Cambridge, 2018.


\bibitem{Tseng}
P. Tseng, Solving h-horizon, stationary Markov decision problems in time proportional to log(h), Oper. Res. Lett. 9 (1990) 287-297.

\bibitem{Ye}
Y. Ye, The simplex and policy-iteration methods are strongly polynomial for the Markov decision problem with a fixed discount rate, Math. Oper. Res. 36 (2011) 593-603.
\end{thebibliography}
\end{document}